\documentclass[12pt]{amsart}
\usepackage{amsmath,amssymb,amsthm,hyperref}
\usepackage{graphicx}
\usepackage{tikz}  
 \newtheorem{theorem}{Theorem}

 \newtheorem{lemma}[theorem]{Lemma}
 \newtheorem{proposition}[theorem]{Proposition}
 
 {\theoremstyle{definition}
 \newtheorem{definition}{Definition}
 \newtheorem{remark}{Remark}

 }

\newcommand{\N}{\ensuremath{\mathbb N}} %natural numbers
\newcommand{\R}{\ensuremath{\mathbb R}} %real numbers
\newcommand{\e}{\varepsilon}
\newcommand{\ep}{\varepsilon}

\newcommand \dimh{\mbox{Dim}_H}
\newcommand \sS{\mathcal{S}}

\title[Sojourn times for fractional Brownian motion]{ Sojourn time dimensions of \\ fractional Brownian motion}

\author{Ivan Nourdin}
\address{Ivan Nourdin, Universit\'e du Luxembourg, 
Unit\'e de Recherche en Math\'ematiques,
Maison du Nombre,
6 avenue de la Fonte,
L-4364 Esch-sur-Alzette,
Grand Duch\'e du Luxembourg}
\email{ivan.nourdin@uni.lu} 
\thanks{}
\author{Giovanni Peccati}
\address{Giovanni Peccati, Universit\'e du Luxembourg, 
Unit\'e de Recherche en Math\'ematiques,
Maison du Nombre,
6 avenue de la Fonte,
L-4364 Esch-sur-Alzette,
Grand Duch\'e du Luxembourg}
\email{giovanni.peccati@uni.lu} 
\thanks{}
\author{St\'ephane Seuret }
\address{St\'ephane Seuret, Universit\'e Paris-Est, LAMA (UMR 8050),  UPEMLV, UPEC, CNRS, F-94010, Cr\'eteil, France}
\email{seuret@u-pec.fr }
  \thanks{S. Seuret thanks the RMATH at University of Luxembourg for its support  during his stay in 2017-2018}

\date{\today}
\begin{document}
\maketitle

\medskip

\begin{abstract}
We describe the size of the sets of sojourn times $E_\gamma =\{t\geq 0: |B_t|\leq t^\gamma\}$ associated with a fractional Brownian motion $B$ in terms of various large scale dimensions. \\
{ {\it Keywords: }  Sojourn time; logarithmic density; pixel density; macroscopic Hausdorff dimension; fractional Brownian motion.}
\end{abstract}

%%%%%%%%%%%%%%%%%%%%%%%%%
%%%%%%%%%%%%%%%%%%%%%%%%%
%%%%%%%%%%%%%%%%%%%%%%%%%
%%%%%%%%%%%%%%%%%%%%%%%%%
%%%%%%%%%%%%%%%%%%%%%%%%%
%%%%%%%%%%%%%%%%%%%%%%%%%
%%%%%%%%%%%%%%%%%%%%%%%%%
%%%%%%%%%%%%%%%%%%%%%%%%%
%%%%%%%%%%%%%%%%%%%%%%%%%
%%%%%%%%%%%%%%%%%%%%%%%%%
%%%%%%%%%%%%%%%%%%%%%%%%%
%%%%%%%%%%%%%%%%%%%%%%%%%
%%%%%%%%%%%%%%%%%%%%%%%%%
\section{Introduction }\label{introduction}

Describing the properties of the sample paths of stochastic processes is { one of the leading threads of modern stochastic analysis}: {such a line of research} started with the investigation of the almost sure continuity properties of the {paths} of a {real-valued} Brownian motion $(B_t)_{t\geq 0}$, {such as H\"older continuity}, followed by the important notions of {\it fast and slow points} {introduced} by Taylor. {See e.g. the three classical references \cite{KS, MP, RY} for formal statements, as well as for an historical overview of this fundamental domain}.

A naturally connected question consists in describing the geometric properties of  the graph of $\{(t,B_t):t\geq 0\}$, in terms of {\it box}, {\it packing} and {\it Hausdorff dimensions} {-- see Section \ref{ss:dimensions} for precise definitons}. {In this respect,} the case of the Brownian motion is \cite{taylor53,taylor55,MP} now very well understood, and many { researchers} have tried, often succesfully, to obtain similar results for other widely used classes of processes: fractional Brownian motions and more general Gaussian processes, L\'evy processes, solutions of SDE or SPDE's (see {\cite {pruitt69,pruitttaylor,ayache,yang,khoshxiao,shiehxiao} for instance, and the numerous references therein). {Despite these remarkable efforts, many important questions in this area are almost completely open for future research.} 

Another description of random trajectories was proposed in terms of {\it sojourn times}. The objective is to describe the (asymptotic) proportion of time spent by a stochastic process in a given region. Sojourn times have been studied by many authors (see for instance \cite{CiesielskiTaylor62,Ray63,uchiyama1982,Bermann91} and the references therein) and play a key role in understanding various features of the paths of stochastic processes, especially { those of} Brownian motion.

In this {paper}, we focus on the sojourn times { associated with the paths of a} fractional Brownian motion (FBM) inside the domain $ \{(t,u): t\geq 0 \mbox{ and } |u|\leq t^\gamma\}$, where $\gamma \geq 0$. It is known that, with probability one, after some large time $t$, an FBM $B:=(B_t)_{t\geq 0}$ does not intersect the domain $ \{(t,u): t\geq 0 \mbox{ and } |u|\geq t^{H+\ep}\}$, for every $\ep>0$. {For this reason, in what follows} we restrict the study to the case $\gamma \in [0,H]$, and investigate the sets
\begin{equation}
\label{defegamma}
E_\gamma :=\{t\geq 0: |B_t|\leq t^\gamma\}
\end{equation}
  in terms of various large scale dimensions: the {\it Lebesgue density}, the {\it logarithmic density}, and the macroscopic {\it box} and {\it Hausdorff dimensions}. {A simulation of the set $E_\gamma$ appears in Fig. 1}.
  
 \begin{figure}
  \includegraphics[width=\linewidth]{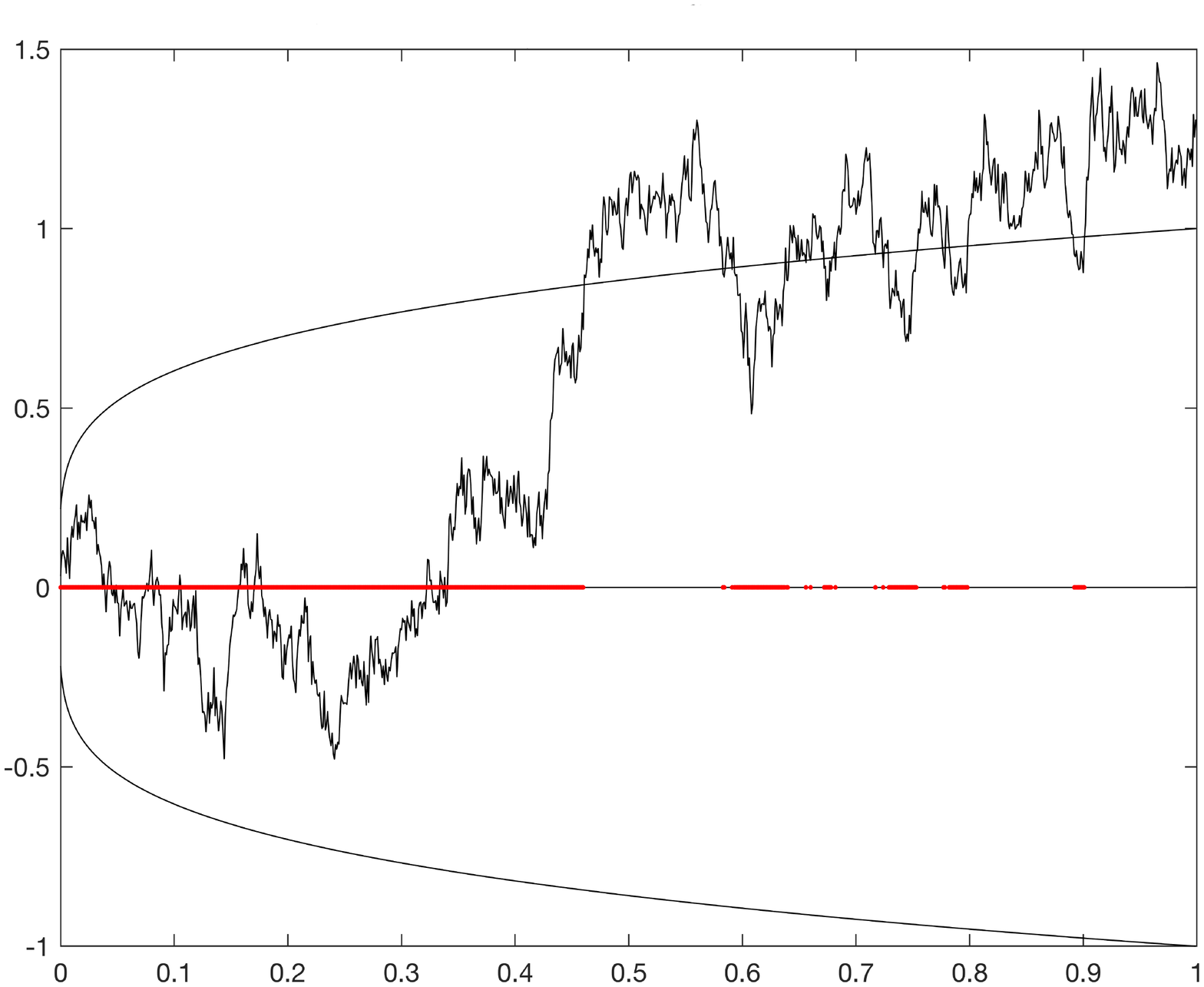}
  \caption{\it The red subset of the real line is a simulation of the set $E_\gamma$, with $\gamma = 0.22$ and $H=0.43$}
\end{figure}

  {The last two notions evoked above} have been introduced in the  late 1980s by Barlow and Taylor (see \cite{barlow1989,barlow1992}), {in order} to {formally define} the fractal dimension of a discrete set. {One of the main motivations for the theory developed in \cite{barlow1989} was e.g. to describe} the asymptotic properties of the trajectory of a random walk on $\mathbb{Z}^2$, {whereas the focus in} \cite{barlow1992} was the computation of the macroscopic Hausdorff dimension of an $\alpha$-stable random walk. Proper definitions are given in the next section. These dimensions have proven to be relevant in other situations, in particular when describing the high peaks of (random) solutions of the stochastic heat equation, see the seminal works of Khoshnevisan, Kim and Xiao \cite{khoshnevisan2015multifractal,khoshnevisanXiao2016}.
    
  \smallskip

The present paper can be seen as a follow-up and a non-trivial extension of  \cite{SX}, where analogous results were obtained in the case of $B$ being a {\em standard} Brownian motion. One of the principal motivations of our analysis is indeed to understand how much the findings of \cite{SX} rely on the {specific features of Brownian motion, such as the (strong and weak) Markov properties, the associated reflection principle, as well as the fine properties of local times}. {While all these features are heavily exploited in \cite{SX}, the novel approach developed in our paper shows that the dimensional analysis of sojourn times initiated in \cite{SX} can be  substantially extended to the non-Markovian setting of a fractional Brownian motion with arbitrary Hurst index. We believe that our techniques might be suitably adapted in order to study sojourn times associated with even larger classes of Gaussian processes or Gaussian fields.} 

\medskip
 
{From now on, every random object considered in the paper is defined on a common probability space $(\Omega, \mathcal{A}, \mathbb{P})$, with $\mathbb{E}$ denoting expectation with respect to $\mathbb{P}$}.

%%%%%%%%%%%%%%%%%%%%%%%%%
%%%%%%%%%%%%%%%%%%%%%%%%%
%%%%%%%%%%%%%%%%%%%%%%%%%
%%%%%%%%%%%%%%%%%%%%%%%%%
%%%%%%%%%%%%%%%%%%%%%%%%%
%%%%%%%%%%%%%%%%%%%%%%%%%
%%%%%%%%%%%%%%%%%%%%%%%%%
%%%%%%%%%%%%%%%%%%%%%%%%%
%%%%%%%%%%%%%%%%%%%%%%%%%
%%%%%%%%%%%%%%%%%%%%%%%%%
%%%%%%%%%%%%%%%%%%%%%%%%%
%%%%%%%%%%%%%%%%%%%%%%%%%
%%%%%%%%%%%%%%%%%%%%%%%%%
\section{Assumptions and main results }\label{results}

%%%%%%%%%%%%%%%%%%%%%%%%%
%%%%%%%%%%%%%%%%%%%%%%%%%
%%%%%%%%%%%%%%%%%%%%%%%%%
%%%%%%%%%%%%%%%%%%%%%%%%%
%%%%%%%%%%%%%%%%%%%%%%%%%
%%%%%%%%%%%%%%%%%%%%%%%%%
\subsection{Densities and dimensions} \label{ss:dimensions}

In what follows, { the symbol} $`{\rm Leb} `$ stands for the one-dimensional Lebesgue measure.
For any set $A$,  we denote by  $|A|$ its cardinality whereas, 
for any subset $E \subset \mathbb{R}^+$,
\begin{align*}
\mbox{pix}(E) = \{n \in \mathbb{N}:  \mbox{dist}(n, E) \leq 1 \}
\end{align*}
is the set of integers that are at distance less than one { from} $E$.  It is clear that 
${\rm Leb}(E)\leq  |\mbox{pix}(E)|$, while the converse inequality does not hold in general.

\medskip

{ We will now describe the main notions and concepts that are used in this paper}, in order to describe the size of the set of sojourn times $E_\gamma= \{t\geq 0: |B_t|\leq t^\gamma\}$. 

\medskip

{ The simplest way of assessing the size of} $E_\gamma$ simply consists in estimating how fast the Lebesgue measure of $E_\gamma\cap [0,t]$ grows with $t$.
  For a general set $E\subset \R^+$, this { yields} the following definition.
  \begin{definition}
{\it
Let $E\subset\R^+$. The \emph{logarithmic density} of $E$ is defined { as}
\begin{eqnarray*}
%\mbox{Den }  E  & =   & \limsup_{n\to+\infty} \frac{  {\rm Leb}(E\cap [0, n])}{n}\\
\mbox{\rm Den}_{\log} E  & =  & \limsup_{n\to+\infty} \frac{\log_2 {\rm Leb}(E\cap [1,2^n])}{n}.
\end{eqnarray*}
}
\end{definition}
This notion will be compared with a similar quantity, { obtained by replacing the Lebesgue measure of a given subset of $E$ by the cardinality of its pixel set. }
%It is referred to as {\em macroscopic density} by Barlow and Taylor, and we adopt here their terminology.
\begin{definition}
{\it
Let $E\subset\R^+$.  The {\em pixel density} of $E$ is defined by
\begin{eqnarray*}
\mbox{\rm Den}_{pix} \, E  & =  & \limsup_{n\to+\infty} \frac{\log_2  |\mbox{\rm pix}(E\cap [1,2^n])|}{n}\\
%\underline{\mbox{\rm Den}}_{pix} \, E  & =  & \liminf_{n\to+\infty} \frac{\log_2  |\mbox{\rm pix}(E\cap [1,2^n])|}{n}.
\end{eqnarray*}
}
\end{definition}
\medskip

The last notion we will deal with is the {\em macroscopic  Hausdorff dimension},   introduced by Barlow and Taylor (as {discussed} above), in order to quantify a sort of ``fractal'' behavior of self-similar structures sitting on infinite lattices.

Following the notations of  \cite{khoshnevisan2015multifractal,khoshnevisanXiao2016}, we consider the   annuli   $\sS_0 = [0,1)$ and $\sS_n = [2^{n-1}, 2^{n})$, for $n\geq 1$. For any $\rho \geq 0$,  any set $E\subset \R^+$ and any $n\in\N^*$, we define
 \begin{align}
\label{def_nu}  
 \nu_\rho^n (E) =  \inf \left\{   \sum_{i=1}^m    \left( \frac{{\rm Leb}(I_i) }{2^{n }}\right)^\rho :  m\geq1, \  I_i\subset \sS_n, \ \  E\cap \sS_n \subset  \bigcup_{i=1}^m I_i  \right\} ,
  \end{align} 
where  $ I_i $ are  non-trivial intervals with  integer boundaries (hence their length is always greater or equal than 1). The infimum is thus taken over a finite number of  finite families  of  non-trivial intervals.
 
%%%%%%%%%%%%%%%%%%%%%%%%%%%%%%%%%%%%%%%%%%
\begin{definition} 
{\it
Let $E\subset \R^+$.  The {\rm macroscopic Hausdorff dimension} of $E$ is defined as
\begin{equation} \label{defdimH}
{\rm Dim}_H E =  \inf\left\{ \rho\ge 0 :  \sum_{n\ge 0} \nu^n_\rho(E)<+\infty   \right\}. 
\end{equation}
}
\end{definition}
%%%%%%%%%%%%%%%%%%%%%%%%%%%%%%%%%%%%%%%%%%
Observe that $\dimh E\in [0,1]$ for any $E\subset \R^+$: indeed,  choosing as covering of $E\cap \mathcal{S}_n $ the intervals of length 1 partitioning $\mathcal{S}_n$, we get  $\nu_{1+\ep}^n(E) \leq 2^{-n\ep }$ for any $\ep>0$, so that $\sum_{n\e 0} \nu^n_{1+\ep}(E)<+\infty  $.

The macroscopic Hausdorff dimension ${\rm Dim}_H E$ of $E\subset \mathbb{R}^+$ does not depend on    its  bounded subsets, since the series in \eqref{defdimH} converges if and only if its tail series converges.   In particular, every bounded set $E$ has a macroscopic Hausdorff dimension equal to zero - the converse is not true, for instance $\dimh \bigcup_{n\geq 1} \{2^n\} = 0$.

Observe also that the local structure of $E$ does not really influence the value of ${\rm Dim}_H (E)$, since the "natural" scale at which $E$ is observed  is 1. 

The value of $\dimh E$ describes the asymptotic distribution of $E\subset \R^+$ on $\R^+$. The difference between $\dimh E$  and the previously introduced dimensions is  that while $\mbox{\rm Den}_{pix} \, E $ (or $\mbox{\rm Den}_{log} \, E $)  only counts the number of points of $E\cap \mathcal{S}_n$ (or, equivalently, {measures} $E \cap [1, 2^n]$), {the quantity} $\dimh E$ takes into account the geometry of the set $E$, {in particular} by considering the most efficient covering of $E\cap \mathcal{S}_n$. For instance, as an intuition, the value of $\nu_n^\rho(E)$ is large when  all the points of $E\cap \mathcal{S}_n$    are more or less uniformly distributed in $\mathcal{S}_n$, while it is much smaller when these points are all located in the same region (in that case,  one large interval is the best possible covering).

Standard inequalities { exploited in our paper} are ({ see} \cite{barlow1989,khoshnevisan2015multifractal})
 \begin{equation}
 \label{ineq_dimensions}
 {\rm Dim}_H E \leq \mbox{\rm Den}_{pix } \,  E  \ \mbox{ and } \ \mathrm{Den}_{\log} E \leq \mbox{\rm Den}_{pix } \,  E.
 \end{equation}
These inequalities are strict in general, in particular the first one will be strict for the sets we focus on in this paper.

 %%%%%%%%%%%%%%%%%%%%%%%%%
%%%%%%%%%%%%%%%%%%%%%%%%%
%%%%%%%%%%%%%%%%%%%%%%%%%
%%%%%%%%%%%%%%%%%%%%%%%%%
%%%%%%%%%%%%%%%%%%%%%%%%%
%%%%%%%%%%%%%%%%%%%%%%%%%
\subsection{Fractional Brownian motion} \label{sec22}

Throughout the paper, $B=(B_t)_{t\geq 0}$ denotes a one-dimensional 
fractional Brownian motion (FBM) of index $H\in(0,1)$.
This means that $B$ is a { continuous Gaussian process}, centered, self-similar of index $H$, and with stationary increments. All these properties {(in particular, the fact that one can always select a continuous modification of $B$)} are simple consequences of the following expression for its covariance function $R$:
$$
R(u,v)=\mathbb{E}[B_uB_v]=\frac12\big(u^{2H}+v^{2H}-|v-u|^{2H}\big).
$$
{ One can easily check that}
\begin{equation}\label{I}
 I:= \iint_{[0,1]^2} \frac{du\, dv}{\sqrt{R(u,u)R(v,v)- R(u,v)^2}} <+\infty.
\end{equation}
{ By virtue of this fact}, the local time $(L_t^x)_{x\in\R,t\geq 0}$ associated with $B$ is well defined in $L^2(\Omega)$ by the following { integral relation}:
\begin{equation}
 \label{deflocal}
 L_t^x = \frac{1}{2\pi}  \int_\mathbb{R} dy\,e^{-iyx} \int_s^t du\,e^{iyB_u},
\end{equation}
see e.g. \cite{Berman69}.
For each $t$, the local time { $x\mapsto L_t^x$} is the density of the occupation measure $\mu_t(A) = {\rm Leb} \{s\in [0,t]: B_s\in A\}$ associated with $B$. Otherwise stated, one has that $L_t = \frac{d\mu_t}{d {\rm Leb}}$ .

A last property that we will need { in order to conclude our proofs}, and that is  an immediate consequence of the Volterra representation of $B$, is that  the natural  filtration associated with FBM
is Brownian. By this, we mean
that there exists a standard Brownian motion $(W_u)_{u\geq 0}$ defined on the same probability space than $B$ such that its filtration satisfies 
\begin{equation}\label{W}
\sigma\{B_u:u\leq t\}  \ \subset \ \sigma\{W_u:u\leq t\}.
\end{equation}
for all $t>0$.

 %%%%%%%%%%%%%%%%%%%%%%%%%
%%%%%%%%%%%%%%%%%%%%%%%%%
%%%%%%%%%%%%%%%%%%%%%%%%%
%%%%%%%%%%%%%%%%%%%%%%%%%
%%%%%%%%%%%%%%%%%%%%%%%%%
%%%%%%%%%%%%%%%%%%%%%%%%%
\subsection{Our results} { Let the notation of the previous sections prevail (in particular $B$ denotes a FBM of index $H\in (0,1)$).} The first result { proved in this paper} concerns the logarithmic and macroscopic densities of the sojourn times $E_\gamma$, as defined in \eqref{defegamma}.
\begin{theorem}
\label{th1}
Fix $\gamma \in [0, H)$. Then
\begin{equation}
 \label{equal1}
 \mbox{\rm Den}_{pix } \,  E_\gamma  = \mbox{\rm Den}_{\log} E_\gamma = \gamma+1-H\quad\mbox{a.s.}
\end{equation}
\end{theorem}
 
\medskip 
 
Our second theorem deals with the macroscopic Hausdorff dimension of all sets $E_\gamma$.

\begin{theorem}
\label{th2}
Fix $\gamma \in [0, H)$. 
Then
\begin{equation}
 \label{up2}
  \mbox{\rm Dim}_{H} E_\gamma   =  1-H\quad\mbox{a.s.}
\end{equation}
\end{theorem}

%\textcolor{red}{- une question: peut-on avoir le resultat pour tous les $\gamma$ en meme temps? c'est clair que c'est vrai pour la majoration, mais c'est la minoration qui pose probleme dans les deux cas} {\bf Ivan dit: j'ai d\'ej\`a essay\'e, mais en vain :-(}
%
%\textcolor{red}{- une 2eme question: peut-on faire $\gamma=H$ et demontrer qu'alors $ \mbox{\rm Dim}_{H} E_H   =  1$?} {\bf Ivan dit: j'ai regard\'e un peu, mais je ne vois pas comment faire -- car d\`es que $\gamma=H$, on ne peut plus se servir d'une puissance n\'egative de 2 pour faire converger les s\'eries qui nous int\'eressent.}

The fact that the macroscopic box and Hausdorff dimension differ asserts that the trajectory enjoys some specific geometric properties. This can be interpreted by the fact that the set $E_\gamma$ is not uniformly distributed (if it were, then both dimensions would coincide), which relies on the intuition that the trajectory of an FBM  does not fluctuate too rapidly from one region to the other.

\medskip 
 
Actually,  the lower bound for the dimension $\mbox{\rm Dim}_{H} E_\gamma   \geq  1-H $ in Theorem \ref{th2} will follow from the next { statement}, which  evaluates the dimension of the level sets
\begin{equation}
\label{deflevel}
\mathcal{L}_x:=\{t: B_t=x\}
\end{equation}
and which is of independent interest.

\begin{theorem}
\label{th3}
Fix $x \in \mathbb{R}$. Then
\begin{equation}
 \label{up3}
  \mbox{\rm Dim}_{H} \mathcal{L}_x   =  1-H\quad\mbox{a.s.}
\end{equation}
\end{theorem}

The {connection between Theorem \ref{th2} and Theorem \ref{th3} can be heuristically understood by observing that, when $t$ is large { and for a fixed $x$}}, { the relation} $t\in \mathcal{L}_x $ implies that $t\in E_\gamma$, and therefore $\mbox{\rm Dim}_{H} E_\gamma   \geq \mbox{\rm Dim}_{H} \mathcal{L}_x $, {owing to the fact that, as explained above, the quantity $\mbox{\rm Dim}_{H} A$ does {\it not} depend on the bounded subsets of a given $A\subset \R_+$}. 
 
%%%%%%%%%%%%%%%%%%%%%%%%%
%%%%%%%%%%%%%%%%%%%%%%%%%
%%%%%%%%%%%%%%%%%%%%%%%%%
%%%%%%%%%%%%%%%%%%%%%%%%%
%%%%%%%%%%%%%%%%%%%%%%%%%
%%%%%%%%%%%%%%%%%%%%%%%%%
%%%%%%%%%%%%%%%%%%%%%%%%%
%%%%%%%%%%%%%%%%%%%%%%%%%
%%%%%%%%%%%%%%%%%%%%%%%%%
%%%%%%%%%%%%%%%%%%%%%%%%%
%%%%%%%%%%%%%%%%%%%%%%%%%
%%%%%%%%%%%%%%%%%%%%%%%%%
%%%%%%%%%%%%%%%%%%%%%%%%%
\section{Proof of Theorem  \ref{th1} : values  of $\mbox{\rm Den}_{pix } \,  E_\gamma  $ and $ \mbox{\rm Den}_{\log} E_\gamma $}
\label{proof1}

In what follows, $C>0$ always denotes a constant whose value is immaterial and may change { from one line to the other}.

 %%%%%%%%%%%%%%%%%%%%%%%%%
%%%%%%%%%%%%%%%%%%%%%%%%%
%%%%%%%%%%%%%%%%%%%%%%%%%
%%%%%%%%%%%%%%%%%%%%%%%%%
%%%%%%%%%%%%%%%%%%%%%%%%%
%%%%%%%%%%%%%%%%%%%%%%%%%
\subsection{Upper bounds} 

Recalling the second part of \eqref{ineq_dimensions}, it is enough to find an upper bound for $\mbox{\rm Den}_{pix } \,  E_\gamma  $, which will also be an upper bound for $\mathrm{Den}_{\log} E_\gamma$.

\medskip

Fix $\gamma\in (0,H)$,  and consider  $ \mbox{\rm Den}_{pix } \,  E_\gamma $. 
First,  we observe that
\begin{eqnarray*}
 \mathbb{E}(|\mbox{pix}(E_\gamma)\cap [1,2^n] |)  &=& \sum_{m=1}^{2^n} \mathbb{P}( \exists s\in [m-1,m+1], \ |B_s|\leq s^\gamma)\\
 &=& \sum_{m=1}^{2^n} \mathbb{P}( \exists s\in [1-\frac1m,1+\frac1m], \ |B_s|\leq s^\gamma m^{\gamma-H})\\
 &\leq& \sum_{m=1}^{2^n} (A^-_{1/m}+A^+_{1/m})
 \end{eqnarray*}
 where
 \begin{eqnarray*}
 A^-_\ep&:=& \mathbb{P}( \exists s\in [1-\ep,1], \ |B_s|\leq \ep^{H-\gamma})\\
 A^+_\ep&:=&\mathbb{P}( \exists s\in [1,1+\ep], \ |B_s|\leq 2 \ep^{H- \gamma}).
 \end{eqnarray*}

 \begin{lemma}
 \label{lemkey}
For  every $\e$ small enough ,
\begin{equation}
\label{eqmaj1}
 \max(A^-_\ep, A^+_\ep)   \leq  3 \ep^{H-\gamma}.
\end{equation}
 \end{lemma}
 
 \begin{proof}
 Let us consider $A^-_\ep$ first. We have
 \begin{eqnarray*}
 A^-_\ep&\leq&  \mathbb{P}( |B_1|\leq 2\ep^{H-\gamma})\\
 &&+ \mathbb{P}( \exists s\in [1-\ep,1], \ |B_s-B_1|\geq \ep^{H-\gamma}).
 \end{eqnarray*}
 The term $\mathbb{P}( |B_1|\leq 2\ep^{H-\gamma})$ is easily  bounded by $C\ep^{H-\gamma}$, so let us concentrate
 on the term $\mathbb{P}( \exists s\in [1-\e,1], \ |B_s-B_1|\geq \e^{H-\gamma})$.
   Set $X_s=B_{1}-B_{1-s}$, $s\in [0,1]$. Observe that $X$ is also a FBM. We have
  \begin{eqnarray*}
 &&\mathbb{P}( \exists s\in [1-\e,1], \ |B_s-B_1|\geq \e^{H-\gamma}) \\
 &=&\mathbb{P}( \exists s\in [0,1], \ |X_{\e s}|\geq \e^{H-\gamma})
 =\mathbb{P}( \exists s\in [0,1], \ |X_{ s}|\geq \e^{-\gamma})\\
&  =&\mathbb{P}( \sup_{s\in [0,1]}|X_{ s}|\geq \e^{-\gamma})
 \leq 2\,\mathbb{P}( \sup_{s\in [0,1]} X_{ s}\geq \e^{-\gamma})
 \end{eqnarray*}
 where last inequality makes use { of the fact} that $X\overset{\rm law}{=}-X$.
 It is well-known that, {by virtue of the Borell and Tsirelson-Ibragimov-Sudakov inequalties (see e.g \cite[Section 2.1]{AT})},
 setting $\alpha=\mathbb{E}\left[\sup_{[0,1]} B\right]$ and because $\mathbb{E}[B_s^2]=s^{2H}\leq 1$ for all $s\in[0,1]$:
 \begin{equation}\label{B-TIS}
 P(\sup_{[0,1]}X\geq u)\leq e^{-\frac{(u-\alpha)^2}{2}},\quad u\geq 0.
 \end{equation}
 (That $\alpha$ is finite is part of the result.)
We deduce that
\begin{eqnarray*}
\mathbb{P}( \exists s\in [1-\e,1], \ |B_s-B_1|\geq \e^{H-\gamma})&\leq& 2\,
e^{-\frac{(\e^{-\gamma}-\alpha)^2}{2}}=O(\e^{\delta}),
\end{eqnarray*}
for every $\delta>0$ when $\e$ becomes small enough. 
 Hence the result. {An analogous argument leads to the same estimate for the set} $A^+_\ep$.
\end{proof}

Going back to $A^-_{1/m}$ and $A^+_{1/m}$, we obtain {from Lemma \ref{lemkey}} that $$\max( A^-_{1/m}, A^+_{1/m}) =O(m^{\gamma-H}).$$ We consequently conclude that 
$$
 \mathbb{E}(|\mbox{pix}(E_\gamma)\cap [1,2^n] |)  \leq 
 \sum_{m=1}^{2^n} (A^-_{1/m}+A^+_{1/m}) = O(2^{n(\gamma+1-H)}).
$$
Choosing  $\rho> \gamma+1-H$, we have
\begin{eqnarray*}
\sum_{n\geq 1}  \mathbb{P}( |\mbox{pix}(E_\gamma)\cap [1,2^n] |>2^{n\rho} ) \leq C\sum_{n\geq 1}  \frac{2^{n(1+\gamma-H)}}{2^{n\rho}} <+\infty.
\end{eqnarray*}
Using the Borel-Cantelli lemma {we infer that}, with probability one, 
$$  |\mbox{pix}(E_\gamma)\cap [1,2^n] | \leq 2^{n\rho} $$
for every large enough integer $n$. Hence $ \mbox{\rm Den}_{pix} E_\gamma \leq \rho$.
Letting $\rho\downarrow \gamma+1-H$ leads to $ \mbox{\rm Den}_{pix} E_\gamma \leq \gamma+1-H$.

\qed

\begin{remark}
We could have proved directly the upper bound for $\mbox{\rm Den}_{\log} E_\gamma $ as follows. 
Introduce 
\begin{equation}
\label{defsgamma}
S_\gamma(t) =  {\rm Leb}\{0\leq s\leq t: |B_s|\leq s^\gamma\}.
\end{equation}
Its expectation 
can be estimated :
\begin{eqnarray}
\nonumber \mathbb{E}(S_\gamma(t)) & = &  \int_0^t \mathbb{P}(|B_s| \leq s^\gamma) \,ds 
 \nonumber  =  \int_0^t \mathbb{P}(|B_1|\leq s^{\gamma-H}) \, ds\\
 \label{eq1_seuret}& \sim & C t^{\gamma+1-H},
\end{eqnarray}
where the Fubini theorem, the self-similarity of $B$ and then the fact that $B_1\sim \mathcal{N}(0,1) $ have been successively used. 
The same argument (based on Borel-Cantelli) as the one used above to conclude that  $ \mbox{\rm Den}_{pix} E_\gamma \leq \gamma+1-H$, allows one to deduce the desired result.
\end{remark}

%we set $I(x,r) = [x,x+r)$ for   $x,r>0$ and 

 %%%%%%%%%%%%%%%%%%%%%%%%%
%%%%%%%%%%%%%%%%%%%%%%%%%
%%%%%%%%%%%%%%%%%%%%%%%%%
%%%%%%%%%%%%%%%%%%%%%%%%%
%%%%%%%%%%%%%%%%%%%%%%%%%
%%%%%%%%%%%%%%%%%%%%%%%%%
\subsection{Lower bounds} 

To obtain the announced lower bounds, we first evaluate the second moment of $S_\gamma(t)$ (defined in \eqref{defsgamma}). We have, with $\Sigma_{u,v}$ denoting the covariance matrix of $(B_u,B_v)$,
\begin{eqnarray*}
 &&\mathbb{E} (S_\gamma(t)^2)\\
  & = & \iint_{[0,t]^2} \mathbb{P}(|B_u|\leq u^\gamma, |B_v|\leq v^\gamma) \, dudv\\
& = & t^2 \iint_{[0,1]^2} \mathbb{P}(|B_u|\leq u^\gamma t^{\gamma-H}, |B_v|\leq v^\gamma t^{\gamma-H} )dudv\\
&=&\frac{t^2}{2\pi} \iint_{[0,1]^2} \frac{dudv}{\sqrt{\det\Sigma_{u,v}}}
\iint_{\R^2}e^{-\frac12
\left(x,y
\right)^T\Sigma_{u,v}^{-1}(x,y)
}
{\bf 1}_{\tiny
\left\{
\begin{array}{l}
|x|\leq u^\gamma t^{\gamma-H}
\\|y|\leq v^\gamma t^{\gamma-H}
\end{array}
\right\}
}
dxdy.
\end{eqnarray*}
Upper bounding $e^{-\frac12\{...\}}$, $u$ and $v$ by 1 and using that
(\ref{I}) is satisfied, we deduce that
\begin{equation}
\label{3}
\mathbb{E} (S_\gamma(t)^2)\leq C\,t^{2\gamma+2-2H}.
\end{equation}
Applying the Paley-Zygmund inequality { together with the estimate} \eqref{eq1_seuret}, we deduce from (\ref{3}) that, for any fixed $0<c<1$, there exists $c'>0$ such that 
\begin{eqnarray*}
 \mathbb{P} ( S_\gamma(2^n) \geq c 2^{n(\gamma+1-H)} )   \geq   (1-c) \frac { \mathbb{E} ( S_\gamma(2^n))^2}{  \mathbb{E} ( S_\gamma(2^n)^2) } \geq c'.
 \end{eqnarray*}
  The Borel-Cantelli lemma ensures that, for infinitely many integers $n$, $S_\gamma(2^n) \geq c 2^{n(\gamma+1-H)} $. {This fact implies that} $  \mathrm{Den}_{\log} E_\gamma   \geq \gamma+1-H$. Finally, using the right inequality in   \eqref{ineq_dimensions}, we directly obtain $\mbox{\rm Den}_{pix } \,  E_\gamma   \geq \gamma+1-H$.

%%%%%%%%%%%%%%%%%%%%%%%%%
%%%%%%%%%%%%%%%%%%%%%%%%%
%%%%%%%%%%%%%%%%%%%%%%%%%
%%%%%%%%%%%%%%%%%%%%%%%%%
%%%%%%%%%%%%%%%%%%%%%%%%%
%%%%%%%%%%%%%%%%%%%%%%%%%
%%%%%%%%%%%%%%%%%%%%%%%%%
%%%%%%%%%%%%%%%%%%%%%%%%%
%%%%%%%%%%%%%%%%%%%%%%%%%
%%%%%%%%%%%%%%%%%%%%%%%%%
%%%%%%%%%%%%%%%%%%%%%%%%%
%%%%%%%%%%%%%%%%%%%%%%%%%
%%%%%%%%%%%%%%%%%%%%%%%%%
\section{Proof of Theorem  \ref{th2}: value of $\dimh E_\gamma$}
\label{proof2}

 %%%%%%%%%%%%%%%%%%%%%%%%%
%%%%%%%%%%%%%%%%%%%%%%%%%
%%%%%%%%%%%%%%%%%%%%%%%%%
%%%%%%%%%%%%%%%%%%%%%%%%%
%%%%%%%%%%%%%%%%%%%%%%%%%
%%%%%%%%%%%%%%%%%%%%%%%%%
\subsection{Upper bound for ${\rm Dim}_H E_\gamma$} 
In what follows, $c>0$ denotes a universal constant whose value is immaterial and may change from one line to another.

Let us fix   $0\le \gamma< H$, as well as  $\eta>0$ (as small as we want). We are going to prove that ${\rm Dim}_H E_\gamma\leq 1-H +\eta$. Letting $\eta$ tend to zero will then give the result.

\smallskip

Fix $\rho> 1-H+\eta$.
Consider for every integer $n\geq 1$ and $i\in \{0,..., \lfloor 2^{n-1}/2^{n\frac{\gamma}H}\rfloor \}$ the times
\[
t_{n,i} = 2^{n-1}+ i 2^{n\frac{\gamma}H} .
\]  
The {collection $t_{n,i}$ generates the intervals} $I_{n,i} = [t_{n,i}, t_{n,i+1})$, { together with the} associated event 
$$\mathcal{E}_{n,i}= \{\exists \, t\in I_{n,i}: |B_t|\leq t^\gamma\},$$

Set  $\ep_{n,i} =  2^{n\frac{\gamma}H}/t_{n,i}$, so that $I_{n,i}=[t_{n,i}, t_{n,i}(1+\ep_{n,i}))$, and observe that    
the ratio between any two of the quantities $2^{n(\frac{\gamma}H-1) }$, $ \ep_{n,i}$  and $ t_{n,i}^{\frac{\gamma}H-1} $ are bounded uniformly with respect to $n$ and $i$. By self-similarity, we have that, when $n$ becomes large,
\begin{align}
\nonumber\mathbb{P}(\mathcal{E}_{n,i}) &= \mathbb{P}(\exists \, \tau\in [1, 1+\ep_{n,i} ] :\,  |B_{\tau \cdot t_{n,i}}|\leq     (\tau\cdot t_{n,i})^\gamma ) \\
\nonumber&= \mathbb{P}(\exists \, \tau\in [1, 1+\ep_{n,i} ] :\,  |B_{t  }|\leq  t_{n,i}^{\gamma-H} \tau^\gamma )
\\
\nonumber&\leq \mathbb{P}(\exists \, \tau\in [1, 1+\ep_{n,i} ] :\,  |B_{t  }|\leq  2 t_{n,i}^{\gamma-H}  )
\\
\nonumber&\leq \mathbb{P}(\exists \, \tau\in [1, 1+\ep_{n,i} ] ,  |B_{t  }|\leq  c\, \ep_{n,i} ^{H} )
\\
\nonumber&\leq \mathbb{P}(\exists \, t\in [1, 1+\ep_{n,i} ] ,  |B_{t  }|\leq  \ep_{n,i} ^{H-\eta} ).
\end{align}
The last { estimate} holds because $\eta$ is a small positive real number and $\ep_{n,i}$ tends to zero. By Lemma \ref{lemkey}, we deduce that $  \mathbb{P}(\mathcal{E}_{n,i})   \leq   c\,  \ep_{n,i} ^{H-\eta}   $ and then
\begin{equation*}
  \mathbb{P}(\mathcal{E}_{n,i})   \leq c\, 2^{n(\gamma-H)\frac{H-\eta}{H}}.
  \end{equation*}
Now observe  that   $\mathcal{E}_{n,i}$ is realized if and only if $E_\gamma\cap I_{n,i} \neq\emptyset$. So, using the intervals $I_{n,i}$ as a covering of $E_\gamma\cap I_{n,i} \neq\emptyset$, we { obtain} from (\ref{def_nu}) that
  \begin{align*}
\mathbb{E}[\nu^n_\rho(E_\gamma)] &\leq  \mathbb{E}\left(\sum_{i=0}^{\lfloor 2^{n-1- n\frac{\gamma}H}\rfloor  }   \left(
\frac{
\mbox{Leb}(I_{n,i})
}{2^n}  \right) ^\rho {\bf 1 \!\!\! 1}_{\mathcal{E}_{n,i} }\right)\\ 
& \leq    2^{ \rho n( \frac{\gamma}H-1)} \sum_{i=0}^{\lfloor 2^{n-1- n\gamma/H}\rfloor  }    \mathbb{P}(\mathcal{E}_{n,i})  \\
& \leq  c\,  2^{  n \frac{ H-\gamma}{H} (1-H+\eta-\rho)} .
\end{align*}
Thus, {the} Fubini Theorem entails {$\mathbb{E}[\sum_{n=1}^\infty \nu_\rho^n(E_\gamma)] < + \infty $ as soon as $\rho >   1-H+\eta$. This implies that for such $\rho$'s,  the sum $\sum_{n=1}^\infty \nu_\rho^n(E_\gamma) $ is finite almost surely. In particular, $\dimh E_\gamma \leq \rho$ for every $\rho >   1-H+\eta$.   { Since such a relation holds for an arbitrary (small) $\rho>0$}, we deduce the desired conclusion.}

 %%%%%%%%%%%%%%%%%%%%%%%%%
%%%%%%%%%%%%%%%%%%%%%%%%%
%%%%%%%%%%%%%%%%%%%%%%%%%
%%%%%%%%%%%%%%%%%%%%%%%%%
%%%%%%%%%%%%%%%%%%%%%%%%%
%%%%%%%%%%%%%%%%%%%%%%%%%
\subsection{Lower bound $\dimh E_\gamma\geq 1-H$} 
\label{proof2-3}

This lower bound  follows from the lower bound in Theorem \ref{th3}, { as proved in} Section \ref{proof3-3}. 

Indeed, assume that $\dimh \mathcal{L}_0 \geq 1-H$, which is an almost sure consequence of Theorem \ref{th3}. 
Obviously    $ \mathcal{L}_0 \subset E_\gamma$, hence  ${\rm Dim}_H E_\gamma \geq 1-H$, which is the desired conclusion.

%%%%%%%%%%%%%%%%%%%%%%%%%
%%%%%%%%%%%%%%%%%%%%%%%%%
%%%%%%%%%%%%%%%%%%%%%%%%%
%%%%%%%%%%%%%%%%%%%%%%%%%
%%%%%%%%%%%%%%%%%%%%%%%%%
%%%%%%%%%%%%%%%%%%%%%%%%%
%%%%%%%%%%%%%%%%%%%%%%%%%
%%%%%%%%%%%%%%%%%%%%%%%%%
%%%%%%%%%%%%%%%%%%%%%%%%%
%%%%%%%%%%%%%%%%%%%%%%%%%
%%%%%%%%%%%%%%%%%%%%%%%%%
%%%%%%%%%%%%%%%%%%%%%%%%%
%%%%%%%%%%%%%%%%%%%%%%%%%
\section{Proof of Theorem  \ref{th3}}
\label{proof3}

 %%%%%%%%%%%%%%%%%%%%%%%%%
%%%%%%%%%%%%%%%%%%%%%%%%%
%%%%%%%%%%%%%%%%%%%%%%%%%
%%%%%%%%%%%%%%%%%%%%%%%%%
%%%%%%%%%%%%%%%%%%%%%%%%%
%%%%%%%%%%%%%%%%%%%%%%%%%
\subsection{A slight modification of $\dimh E$}

In this section, we  use a  slightly modified version of $\nu^n_\rho$ defined as
 \begin{eqnarray} 
 \tilde\nu_\rho^n (E) &=& \inf \left\{    \sum_{i=1}^m    \left( \frac{{\rm Leb}(I_i) }{2^n} \right)^\rho \left |\log _2 \frac{{\rm Leb}(I_i) }{2^n}\right|^{1-\rho} : \right.\nonumber \\
&&\left. \hskip1.3cm m\geq1, \  I_i\subset \sS_n, \ \  E\cap \sS_n \subset  \bigcup_{i=1}^m I_i  \right\}.
\label{defnut} 
  \end{eqnarray} 
The introduction of a logarithm factor makes some computations easier in Section \ref{proof3-3}. The quantities $\nu^n_\rho$ lead to the same notion of dimension. Indeed, it is  easily proved \cite{khoshnevisan2015multifractal,SX} that one can replace $\nu$ by $\tilde\nu$ in \eqref{defdimH}, so that 
\begin{equation}
\label{defdimt}
{\rm Dim}_H E =  \inf\left\{ \rho\ge 0 :  \sum_{n\ge 0} \tilde\nu^n_\rho(E)<+\infty   \right\}.
\end{equation}

 %%%%%%%%%%%%%%%%%%%%%%%%%
%%%%%%%%%%%%%%%%%%%%%%%%%
%%%%%%%%%%%%%%%%%%%%%%%%%
%%%%%%%%%%%%%%%%%%%%%%%%%
%%%%%%%%%%%%%%%%%%%%%%%%%
%%%%%%%%%%%%%%%%%%%%%%%%%
\subsection{Upper bound for $\dimh \mathcal{L}_x$} 
 The argument { exploited in the present section} is comparable to the one used in Section \ref{proof2-3}.
 
Since every levet set $\mathcal{L}_x$ defined by \eqref{deflevel} is ultimately included in $E_\gamma$ for every $\gamma>0$, and since all the dimensions we consider do not depend of any bounded subset of $E_\gamma$,  we easily { obtain} from \eqref{ineq_dimensions} and Theorem \ref{th1}  that ${\rm Dim}_H \mathcal{L}_x \leq 1-H+\gamma$, for any $\gamma>0$. {Letting $\gamma\downarrow0$, one sees that ${\rm Dim}_H \mathcal{L}_x \leq 1-H$.}

 %%%%%%%%%%%%%%%%%%%%%%%%%
%%%%%%%%%%%%%%%%%%%%%%%%%
%%%%%%%%%%%%%%%%%%%%%%%%%
%%%%%%%%%%%%%%%%%%%%%%%%%
%%%%%%%%%%%%%%%%%%%%%%%%%
%%%%%%%%%%%%%%%%%%%%%%%%%
\subsection{Lower bound for $\dimh \mathcal{L}_x$} 
\label{proof3-3}

Let us  now introduce  the random variables 
\begin{equation}
\label{defF}
Y^x_n = \frac{ L^{x2^{nH}}_{2^n} - L^{x2^{nH}}_{2^{n-1}}  } {2^{n(1-H)}  }  \ \ \mbox{ and } \ \  F^x_N:=  \sum_{n= 1}^N   Y^x_n.
\end{equation}
The random sequence $(F^x_N)_{N\ge 1}$ is non-decreasing and we denote by $F^x_\infty$ its limit, i.e. $F^x_\infty = \sum_{n\geq 1} Y^x_n $.

\noindent We remark from the self-similarity of $B$ that $Y^x_n \overset{d}{=} Y^x_0$, see formula (\ref{deflocal}).

\medskip

Let us start with a lemma connecting the r.v. $Y_n^x$ to the macroscopic Hausdorff dimension.

%%%%%%%%%%%%%%%%%%%%%%%%%
\begin{lemma}
\label{lem_minor}
With probability one, there exists a constant $K>0$ such that, for every $x\in \mathbb{R}$ and every $n\geq 1$,
$$\tilde \nu^n_{1-H}(\mathcal{L}_x ) \geq  K^{-1}   Y^x_n.$$
\end{lemma}
 %%%%%%%%%%%%%%%%%%%%%%%%%

 %%%%%%%%%%%%%%%%%%%%%%%%%
\begin{proof}
We start by recalling a key result of Xiao (see \cite[Theorem 1.2]{Xiao_localtimes}), which describes the scaling behavior of the local times of stationary Gaussian processes. For this, let us introduce the random variables 
$$X_n:= \sup_{0\leq t\leq 2^n} \sup_{0\leq h\leq 2^{n-1}} \sup_{x\in \mathbb{R}} \frac{L^x_{t+h}-L^x_t}{h^{1-H}(n-\log_2h)^H}.$$
Self-similarity of $B$ implies 
\begin{eqnarray*}
 X_n 
 &= &\sup_{0\leq t\leq 1} \, \sup_{0\leq h\leq 1/2} \, \sup_{x\in \mathbb{R}} \frac{L^{x2^{nH}}_{2^n(t+h)}-L^{x2^{nH}}_{2^nt}}{(2^nh)^{1-H}(-\log_2h)^H}\\
 &  \overset{d}{=} & \sup_{0\leq t\leq 1} \,  \sup_{0\leq h\leq 1/2} \, \sup_{x\in \mathbb{R}} \frac{L^x_{t+h}-L^x_{t}}{h^{1-H}(-\log_2h)^H}
\end{eqnarray*}
By \cite[Theorem 1.2]{Xiao_localtimes}, with probability one there exists a constant $K>0$ such that 
\begin{equation}
\label{bound1}
\mbox{for every $n\geq 1$, } \ \ X_n \leq K.
\end{equation}
Now fix $x\in \mathbb{R}$, and consider the associated level set $\mathcal{L}_x$ defined by \eqref{deflevel}. Recall the definition  \eqref{defnut} of $ \tilde\nu ^n_{1-H}(\mathcal{L}_x)$.  Choose a covering $(I_i)_{i=1,...,m}$  that minimizes the value in \eqref{defnut}, and set $I_i=[x_i,y_i]$. We observe that
\begin{eqnarray*}
\tilde \nu^n_{1-H}(\mathcal{L}_x ) &=&    \sum_{i=1}^m    \left( \frac{{\rm Leb}(I_i) }{2^n} \right)^{1-H} \left |\log _2 \frac{{\rm Leb}(I_i) }{2^n}\right|^{H} \\
 &=&    \sum_{i=1}^m     \left( \frac{| y_i-x_i | }{2^n} \right)^{1-H} \left |\log_2  \frac{| y_i-x_i | }{2^n}\right|^{H}\\
&\geq & K^{-1} \sum_{i=1}^m    \frac{ L^{x2^{nH}}_{y_i}-  L^{x2^{nH}}_{x_i}   } {2^{n(1-H)}   }
=  K^{-1} \sum_{i=1}^m    \frac{ L^{x2^{nH}} (I_i)   } {2^{n(1-H)}   } \\
&\geq & K^{-1}      \frac{ L^{x2^{nH}}_{2^n} - L^{x2^{nH}}_{2^{n-1}}  } {2^{n(1-H)}   }
\end{eqnarray*}
where \eqref{bound1} has been used to get the first inequality, and the last inequality { holds because} the local time $L_.^x$ increases only on the sets $I_i$ (whose union covers $\mathcal{L}_x\cap \mathcal{S}_n$). This proves the claim.
\end{proof}
 %%%%%%%%%%%%%%%%%%%%%%%%%

%%%%%%%%%%%%%%%%%%%%%%%%%
\begin{remark}
The introduction of $\tilde\nu^n_\rho $ instead of $\nu^n_\rho$ in  \eqref{defdimt} is key in the last sequence of inequalities { displayed in the previous proof}, {allowing us} to use in a relevant way Xiao's result \eqref{bound1}.
\end{remark}
%%%%%%%%%%%%%%%%%%%%%%%%%

Now, using Lemma \ref{lem_minor}, and recalling \eqref{defdimt}, in order to conclude that $\dimh \mathcal{L}_x \geq 1-H$ and Theorem \ref{th3}, it is enough to prove that the series $\sum_{n\geq 1}Y^x_n$ diverges almost surely. This is the purpose of the next proposition.

%%%%%%%%%%%%%%%%%%%%%%%%%
\begin{proposition}
\label{propproba1}
For all
$x\in\R$,
\begin{equation}
\label{proba1}
\mathbb{P} (F^x_\infty =+\infty) =1.
\end{equation}
\end{proposition}
%%%%%%%%%%%%%%%%%%%%%%%%%

%\textcolor{red}{voir si on peut avoir $\mathbb{P} (\forall x, \ F^x_\infty =+\infty) =1.$}

The proof of Proposition \ref{propproba1} makes use of various arguments involving local times, Brownian filtration and Kolmogorov 0-1 law. As a preliminary step, we start with the following lemma,  showing that the previous probability is strictly positive. {Our key argument can be seen as a variation of the celebrated {\it Jeulin's Lemma} \cite[p. 44]{J}, allowing one to deduce the convergence of random series (or integrals), by controlling deterministic series of probabilities.}

%%%%%%%%%%%%%%%%%%%%%%%%%
\begin{lemma}
\label{probapositive}
For every $x\in\R$, one has that
\begin{equation}
\label{eq2s}
\mathbb{P} (F^x_\infty =+\infty)  >0.
\end{equation}
\end{lemma}
 %%%%%%%%%%%%%%%%%%%%%%%%%

%%%%%%%%%%%%%%%%%%%%%%%%%
\begin{proof}

%\textcolor{red}{mentionner plus explicitement Jeulin}
Recalling  \eqref{deflocal} we have, for every $s\leq t$, 
\begin{eqnarray*}
 L^x_t-L^x_s =   \frac{1}{2\pi}  \int_\mathbb{R} dy  \,e^{-iyx} \int_s^t du\, e^{iyB_u}.
\end{eqnarray*}
Using the self-similarity of $B$ through $\mathbb{E}(B_u^2)=u^{2H}\mathbb{E}(B_1) = u^{2H}$, we deduce:
\begin{eqnarray*}
 &&\mathbb{E}(L^x_t-L^x_s)  =     \frac{1}{2\pi}  \int_\mathbb{R} dy\,e^{-iyx}  \int_s^t du\,e^{-\frac12y^2u^{2H}} \\
 &=  &   \frac{1}{2\pi} \int_s^t du \int_\mathbb{R} dy\,e^{-iyx}  \,e^{-\frac12y^2u^{2H}} 
  =  \frac{1}{\sqrt{2\pi}}    \int_s^t e^{-\frac{x^2}{2u^{2H}}} u^{-H}  du.
\end{eqnarray*}
We observe in particular that $ \mathbb{E}(L^x_1-L^x_\frac12) >0$, so  
\begin{equation}\label{pos}
{P}(Y^x_0>0) =\mathbb{P}(L^x_1-L^x_\frac12>0)  >0.
\end{equation}
Now fix $\gamma>0$, and consider the event $A=\{F^x_\infty \leq \gamma\}$. We have, by Fubini,
\begin{eqnarray*}
 \gamma & \geq & \mathbb{E}({\bf 1 \!\!\!1} _A F^x_\infty) = \sum_{n\geq 1} \mathbb{E} ({\bf 1 \!\!\!1}_A Y^x_n)
= \sum_{n\geq 1}\int_0^{+\infty}  \mathbb{P} \left(  A \cap \{Y^x_n>u\} \right)du.
\end{eqnarray*}
Using $\mathbb{P}(A\cap B)\geq (\mathbb{P}(A)-\mathbb{P}(B^c))_+$, we deduce
that 
\begin{eqnarray*}
 \gamma & \geq & \sum_{n\geq 1}\int_0^{+\infty}  (\mathbb{P}(A) -  \mathbb{P}Y^x_n\leq u))_+ du\\
 &=&\sum_{n\geq 1}\int_0^{+\infty}  (\mathbb{P}(A) -  \mathbb{P}(Y^x_0\leq u))_+ du.
\end{eqnarray*}
Since the summand does not depend on $n$, the only possibility is that it is zero, that is,
$$\int_0^{+\infty}  (\mathbb{P}(F^x_\infty \leq \gamma) -  \mathbb{P}(Y^x_0\leq u))_+ du=0.$$ 
This implies,
for almost every $u\geq 0$ and every $\gamma>0$:
$$\mathbb{P}(F^x_\infty \leq \gamma) \leq   \mathbb{P}(Y^x_0\leq u).$$
Letting $\gamma\to +\infty$ together with $u \to 0^+$, and recalling (\ref{pos}), we conclude that 
 $$\mathbb{P} (F^x_\infty =+\infty) \geq \mathbb{P}(Y^x_0> 0) >0,$$
which is exactly { the desired relation} (\ref{eq2s}).
\end{proof}
%%%%%%%%%%%%%%%%%%%%%%%%%

\medskip

It remains us to prove that not only $\mathbb{P} (F^x_\infty =+\infty)$ is strictly positive for every $x$, but in fact it equals 1. {Such a conclusion will follow from the next statement, corresponding to a time-inversion property of FBM}. It can be checked immediately by computing the covariance function of the process $\widetilde{B}$ {introduced below}.

%%%%%%%%%%%%%%%%%%%%%%%%%
\begin{lemma}
\label{lem2}
The reversed time process $\widetilde B$
\begin{equation}
\label{defbtilde}
t\longmapsto \widetilde B_u:=u^{2H}B_{1/u}
\end{equation}
is also a FBM.
\end{lemma}

 %%%%%%%%%%%%%%%%%%%%%%%%%

Let us denote by $\widetilde L^x_t$, $\widetilde Y^x_n $ and $\widetilde F^x_N$ the quantities analogous to  $  L^x_t$, $  Y^x_n $ and $  F^x_N$ defined in \eqref{defF}, but  associated with $\widetilde B$ (see \eqref{defbtilde}) instead of $B$. Obviously, $(L^x_t)_{x\in\R,t\geq 0}$ and $(\widetilde L^x_t)_{x\in\R,t\geq 0}$ have the same law. 
So 
$$F^x_\infty \overset{d}{=} \widetilde F^x_\infty  :=      \sum_{n= 1}^{+\infty}  \ \frac{\widetilde L^{2^{nH}x}_{2^n} - \widetilde L^{2^{nH}x}_{2^{n-1}}  } {2^{n(1-H)}  } .$$
For a fixed integer $n\geq 1$,  we have 
$$\widetilde L^{x2^{nH}}_{2^n} - \widetilde L^{x2^{nH}}_{2^{n-1}}  =\frac{1}{2\pi} \int_{\mathbb{R}} dy \,e^{-iy2^{nH}x}\int_{2^{n-1}}^{2^n} du\, e^{iyu^{2H}B_{1/u}},$$
implying in turn that $\widetilde L^{x2^{nH}}_{2^n} - \widetilde L^{x2^{nH}}_{2^{n-1}} $ is $\sigma\{B_u:u\leq 2^{-(n-1)}\}$-measurable.
As a consequence, for every $M\geq 1$,
$$\sigma\left\{ \widetilde L^{x2^{nH}}_{2^n} - \widetilde L^{x2^{nH}}_{2^{n-1}}  : n\geq M\right\}\subset \sigma\{B_u:u\leq 2^{-(M-1)}\}.$$
The event $\{\widetilde F^x_\infty  =+\infty\}$ does not depend on the first term of the series, so is a tail event. Otherwise stated,
$$\{\widetilde F^x_\infty  =+\infty\} \in \bigcap_{M\geq 1} \sigma\left\{ \widetilde L^{x2^{nH}}_{2^n} - \widetilde L^{x2^{nH}}_{2^{n-1}}  : n\geq M\right\} .$$
Using now (\ref{W}) (with $\widetilde{B}$ instead of $B$), we deduce that
$$\{\widetilde F^x_\infty  =+\infty\} \in \bigcap _{M\geq 1} \sigma\{W_u:u\leq 2^{-M}\},$$
where $W$ is a standard Brownian motion. By the Blumenthal's 0-1 law for $W$, we infer that $\mathbb{P}(\widetilde F^x_\infty  =+\infty)$ is either 0 or 1.  Remembering Lemma \ref{probapositive}, we can conclude that this probability is one, which implies \eqref{proba1} as claimed.\qed

%%%%%%%%%%%%%%%%%%%%%%%%%%%%%%%%%%
%%%%%%%%%%%%%%%%%%%%%%%%%%%%%%%%%%
%%%%%%%%%%%%%%%%%%%%%%%%%%%%%%%%%%
%%%%%%%%%%%%%%%%%%%%%%%%%%%%%%%%%%
%%%%%%%%%%%%%%%%%%%%%%%%%%%%%%%%%%
%%%%%%%%%%%%%%%%%%%%%%%%%%%%%%%%%%
%%%%%%%%%%%%%%%%%%%%%%%%%%%%%%%%%%
%%%%%%%%%%%%%%%%%%%%%%%%%%%%%%%%%%
%%%%%%%%%%%%%%%%%%%%%%%%%%%%%%%%%%
%%%%%%%%%%%%%%%%%%%%%%%%%%%%%%%%%%
%%%%%%%%%%%%%%%%%%%%%%%%%%%%%%%%%%

\bigskip
 
\noindent {\bf Acknowledgement}. We thank Ciprian Tudor for suggesting that it may be useful to rely on the representation (\ref{deflocal}) of the local time, and to Yimin Xiao {for useful discussions around the (wide) literature generated by his seminal work
\cite{Xiao_localtimes}. }
 
\bibliographystyle{plain}

\end{document}